\documentclass[12pt]{amsart}

\usepackage{amsmath,amsfonts,amssymb,amsthm}
\usepackage{hyperref}

\title[Sufficient conditions for additivity]{Sufficient conditions for Strassen's additivity conjecture}
\date{\today}
\author{Zach Teitler}
\address{Zach Teitler\\
1910 University Drive\\
Department of Mathematics\\
Boise State University\\
Boise, ID 83725-1555 \ USA}
\email{zteitler@boisestate.edu}
\thanks{This work was supported by a grant from the Simons Foundation (\#354574, Zach Teitler).}


\newtheorem{theorem}{Theorem}[section]
\newtheorem{maintheorem}{Theorem}
\newtheorem{proposition}[theorem]{Proposition}
\newtheorem{corollary}[theorem]{Corollary}

\theoremstyle{definition}

\newtheorem{definition}[theorem]{Definition}
\newtheorem{question}[maintheorem]{Question}

\theoremstyle{remark}

\newtheorem{remark}[theorem]{Remark}
\newtheorem{example}[theorem]{Example}

\newcommand{\bx}{\mathbf{x}}

\newcommand{\bpartial}{\pmb{\partial}}

\newcommand{\bbC}{\mathbb{C}}
\newcommand{\bbP}{\mathbb{P}}

\DeclareMathOperator{\rk}{rk}
\DeclareMathOperator{\brk}{brk}
\DeclareMathOperator{\crk}{crk}
\DeclareMathOperator{\Derivs}{Derivs}
\DeclareMathOperator{\codim}{codim}
\DeclareMathOperator{\mult}{mult}
\DeclareMathOperator{\Span}{span}

\DeclareMathOperator{\minrank}{minrank}
\DeclareMathOperator{\Zeros}{Zeros}

\newcommand{\defining}[1]{\emph{#1}}

\begin{document}

\begin{abstract}
We give a sufficient condition for the strong symmetric version of Strassen's additivity conjecture:
the Waring rank of a sum of forms in independent variables is the sum of their ranks,
and every Waring decomposition of the sum is a sum of decompositions of the summands.
We give additional sufficient criteria for the additivity of Waring ranks
and a sufficient criterion for additivity of cactus ranks and decompositions.
\end{abstract}

\maketitle

The \defining{Waring rank} $\rk(F)$ of a homogeneous polynomial $F$ of degree $d$ is the least number of terms $r$
in an expression for $F$ as a linear combination of $d$th powers of linear forms, $F = c_1 \ell_1^d + \dotsb + c_r \ell_r^d$.
Higher Waring rank corresponds to higher complexity.
For general introductions see, for example, \cite{Geramita}, \cite{MR1735271}, \cite{Teitler:2014gf};
for discussion of applications see, for example, \cite{MR2447451}, \cite{MR2865915}.
We work over $\bbC$, that is, our homogeneous polynomials have complex coefficients.
Since we work over $\bbC$ we may rescale each $\ell_i$ to make the scalars $c_i$ unnecessary.
An expression $F = \ell_1^d + \dotsb + \ell_r^d$ is called a \defining{power sum decomposition} of $F$.
A power sum decomposition with $r = \rk(F)$ is called a \defining{Waring decomposition}.
Homogeneous polynomials of degree $d$ are henceforth called \defining{forms of degree $d$},
or briefly \defining{$d$-forms}.

We write $F(\bx)$ for a polynomial in the tuple of variables $\bx = (x_1,\allowbreak \dotsc, \allowbreak x_n)$.
Suppose that $F_1(\bx_1), \dotsc, F_k(\bx_k)$ are forms of degree $d$ in independent tuples $\bx_1,\dotsc,\bx_k$
and $F(\bx_1,\dotsc,\bx_k) = F_1(\bx_1) + \dotsb + F_k(\bx_k)$.
Then clearly $\rk(F) \leq \rk(F_1) + \dotsb + \rk(F_k)$,
as adding together Waring decompositions (minimal power sum decompositions) of the $F_i$'s gives a power sum decomposition of $F$.
\begin{question}\label{q: rank additive}
With notation as above, is $\rk(F) = \rk(F_1) + \dotsb + \rk(F_k)$?
\end{question}
\noindent
We give sufficient conditions for this equality to hold.

Strassen's additivity conjecture \cite{MR0521168} asserts a similar statement for tensors.
Waring rank of homogeneous forms is equivalent to symmetric rank of symmetric tensors.
So this question is asking for a symmetric version of Strassen's additivity conjecture.

One might ask if, further,
every Waring decomposition for $F$ is a sum of Waring decompositions of the $F_i$,
each using only the variables $\bx_i$.
Explicitly, in a Waring decomposition $F = \sum \ell_j^d$,
we ask whether it is the case that
each linear form $\ell_j$ involves variables from only one tuple $\bx_i$.
If this holds, the $\ell_j^d$ that involve variables from $\bx_i$ must necessarily sum to $F_i$.
It fails for quadratic forms---for example, $x^2+y^2 = (cx+sy)^2+(sx-cy)^2$ whenever $c^2+s^2=1$---but
the question remains open for higher degree forms.

\begin{question}\label{q: decompositions additive}
With notation as above, is every Waring decomposition of $F$
given by a sum of Waring decompositions of the $F_i$'s?
\end{question}

The symmetric version of Strassen's additivity conjecture has recently been shown to hold in a number of cases.
When all of the $F_i$ are monomials a positive answer to Question~\ref{q: rank additive} is given in~\cite{MR2966824}.
When all but one of the $F_i$ has the simple form $x_i^d$,
or when the number of summands is $k=2$ and each $F_1, F_2$ is a binary form,
Question~\ref{q: rank additive} has a positive answer \cite{MR3320211}.
These cases and many more are treated in a more uniform manner in \cite{MR3783793},
which introduces the notion of ``$e$-computability'' and shows it provides a sufficient condition for
Question~\ref{q: rank additive} to have a positive answer.
Very recently a positive answer to Question~\ref{q: decompositions additive} for certain forms $F_i$
has been given in \cite[Theorem 4.6, Theorem 5.1]{MR3658727}.

We give a simple sufficient condition for the additivity of Waring rank and Waring decompositions to hold,
that is, for Question~\ref{q: rank additive} and Question~\ref{q: decompositions additive} to both have positive answers.
The condition is surprisingly simple: it is just that equality should hold in a certain well-known lower bound for Waring rank,
the catalecticant bound.

The catalecticant bound for Waring rank has been known since the 19th century, see discussion and a precise statement below.
However it seems to never have been noticed that when equality holds, a consequence is additivity of ranks as in Strassen's conjecture.

Equality in the catalecticant bound is a special condition,
but it holds for various classes of forms, including general binary forms,
general forms of low rank, general quartic forms in $n \leq 6$ variables,
and general plane conics, quartics, sextics, and octics; see \textsection\ref{section: catalecticant bound}.

Our first main result, involving Waring rank and the catalecticant bound, is described in Section~\ref{section: additivity of Waring rank}.
The underlying idea is very simple and yields a number of immediate generalizations, described in the subsequent sections.
In Section~\ref{section: further bounds} we consider other bounds for Waring rank, obtaining additional sufficient conditions
for a positive answer to Question~\ref{q: rank additive}.
In particular we describe in some detail a very recently discovered lower bound due to Carlini--Guo--Ventura.
In Section~\ref{section: cactus rank} we return to the catalecticant bound, which is also a bound for cactus rank.
We obtain a sufficient condition for a positive answer to the analogues
of Question~\ref{q: rank additive} and Question~\ref{q: decompositions additive} for cactus rank.

The Waring ranks of monomials and sums of monomials in independent variables are known \cite{MR2966824},
cactus ranks of monomials are known \cite{MR2842085},
and border ranks of monomials have been found recently as well \cite{Oeding:2016kl}.
However cactus ranks and border ranks of sums of monomials in independent variables are not known.
We determine these in special cases
in Example \ref{example: concentrated monomials cactus rank} and Example \ref{example: concentrated monomials border rank}.

\section{Additivity of Waring rank}\label{section: additivity of Waring rank}

The general strategy for giving a positive answer to Question~\ref{q: rank additive} is very simple.
Suppose that $A$ and $B$ are functions of homogeneous forms satisfying the following conditions:
\begin{enumerate}
\item $A(F) \geq B(F)$ for every $F$,
\item $A$ is subadditive, meaning $A(F+G) \leq A(F) + A(G)$ for all $d$-forms $F$ and $G$, and
\item $B$ is additive on forms in independent variables, meaning $B(F+G) = B(F)+B(G)$ when $F$ and $G$ are $d$-forms in independent variables.
\end{enumerate}
Given all this, when $F_1,\dotsc,F_k$ are $d$-forms in independent variables such that $A(F_i) = B(F_i)$ for each $i$,
then $A$ is also additive on the $F_i$: $A(\sum F_i) = \sum A(F_i)$.

The notion of ``$e$-computability'' fits into this framework, see \cite[Corollary 3.4, Definition 3.5]{MR3783793}.

In this section $A$ is Waring rank and $B$ is the catalecticant bound, described next.
In Section~\ref{section: further bounds} we keep Waring rank as $A$, and vary $B$.
In Section~\ref{section: cactus rank} we return to the catalecticant bound for $B$, and change $A$ to cactus rank.

\begin{remark}
See \cite{Teitler:2014gf} for an exposition of lower bounds for generalized ranks,
from which the interested reader may formulate sufficient conditions for Strassen-like additivity results
in many more cases.
\end{remark}

\subsection{Catalecticant bound}\label{section: catalecticant bound}

For a form $F$ of degree $d$,
we denote by $\Derivs(F)$ the vector space spanned by the partial derivatives of $F$ of all orders,
including $F$ itself.
It is a graded finite-dimensional vector space;
let $\Derivs(F)_a$ denote the $a$th graded piece, that is, the space of $a$-forms spanned by the $(d-a)$th derivatives of $F$.
Recall the following very well-known bound for Waring rank:
\begin{proposition}[Classical]
For every $a$, $0 \leq a \leq d$, we have $\rk(F) \geq \dim \Derivs(F)_a$.
\end{proposition}
\noindent
Indeed, if $F = \sum c_i \ell_i^d$, then $\partial F / \partial x_j = \sum d c_i m_{i,j} \ell_i^{d-1}$ for some constants $m_{i,j}$,
explicitly $m_{i,j}$ is the coefficient of $x_j$ in $\ell_i$.
By induction $\Derivs(F)_a$ is contained in the span of $\ell_1^a,\dotsc,\ell_r^a$, so has dimension at most $r$.
This bound dates back to the 19th century; it is sometimes called the catalecticant bound for Waring rank.
For more on this, including the name ``catalecticant'',
see, for example, \cite[\textsection 2.1]{Teitler:2014gf}, \cite[pg.~49--50]{MR1096187}, \cite[Lecture~11]{Geramita},
\cite{Miller:2013lo}.

We are interested in the case that equality holds.
Equality in the above bound is a fairly special condition, but it does occur.
We do not classify all cases where equality holds, but we list several here.

\begin{example}
Fix $n$ and $d$.
For $1 \leq r \leq n$ the rank-$r$ Fermat-type polynomial $F = x_1^d + \dotsb + x_r^d$
has rank $\rk(F) = r = \dim \Derivs(F)_a$ for all $1 \leq a \leq d-1$.
In particular this holds for every quadratic form for $a = 1$.
\end{example}

\begin{example}
By well-known facts about binary forms (see, for example, \cite[\textsection1.3]{MR1735271})
if $F = F(x,y)$ is a binary form of degree $d$ and rank $r = \rk(F) \leq \left\lfloor \frac{d+2}{2} \right\rfloor$,
then $\rk(F) = \dim \Derivs(F)_a$ for $r-1 \leq a \leq d+1-r$.
In particular, if $F = F(x,y)$ is a general binary form of degree $d$ then
$\rk(F) = \left\lfloor \frac{d+2}{2} \right\rfloor = \dim \Derivs(F)_{\left\lfloor \frac{d}{2} \right\rfloor}$.
\end{example}

\begin{remark}
We say that a statement holds for a \defining{general} form of some type if the statement holds for every element in
a Zariski open and dense subset of the set of forms of that type.

For example, a general form of degree $d$ in $n$ variables $x_1,\dotsc,x_n$
is an element of a Zariski open and dense subset of the vector space
parametrizing all forms of degree $d$ in the $n$ variables given.
It is equivalent to say that the form has general coefficients.
\end{remark}

Recall that if $\ell_1,\dotsc,\ell_r$ are general linear forms in $n$ variables
and $F = \ell_1^d + \dotsb + \ell_r^d$ then for $0 \leq a \leq d$ we have
\[
  \dim \Derivs(F)_a = \min\left\{ \binom{n+a-1}{n-1} , \binom{n+d-a-1}{n-1} , r \right\},
\]
see \cite[Lemma 1.17]{MR1735271}.
(Here, \defining{general} means that the $r$-tuple $(\ell_1,\dotsc,\ell_r)$ is an element
of a Zariski open and dense subset of the $r$th Cartesian power of the vector space of linear forms.)
\begin{proposition}
Fix integers $n, d, a$ with $1 \leq a \leq d-1$.
Let $r$ be an integer, $1 \leq r \leq \min\left\{ \binom{n+a-1}{n-1} , \binom{n+d-a-1}{n-1} \right\}$.
Let $F$ be a general form of rank $r$, of degree $d$ in $n$ variables.
Then $\rk(F) = \dim \Derivs(F)_a$.
\end{proposition}
\begin{proof}
Immediate.
\end{proof}

Recall the Alexander--Hirschowitz theorem (see, for example, \cite[Corollary 1.62]{MR1735271}):
If $F$ is a general form of degree $d$ in $n$ variables then $\rk(F) = \lceil \frac{1}{n} \binom{n+d-1}{n-1} \rceil$,
with the following exceptions.
If $d=2$, then $\rk(F) = n$ (instead of $\lceil (n+1)/2 \rceil$);
if $(n,d) = (3,4), (4,4), (5,4), (5,3)$ then $\rk(F) = 6$, $10$, $15$, $8$, respectively (instead of $5$, $9$, $14$, $7$, respectively).

And recall that if $F$ is a general form of degree $d$ in $n$ variables
then for $0 \leq a \leq d$ we have
\[
  \dim \Derivs(F)_a = \min\left\{ \binom{n+a-1}{n-1} , \binom{n+d-a-1}{n-1} \right\} ,
\]
see \cite[Proposition 3.12]{MR1735271}.

\begin{example}
Equality holds in the catalecticant bound with $a=2$ for general quartics in $n \leq 6$ variables.
That is, if $F$ is a general form of degree $4$ in $n \leq 6$ variables, then $\rk(F) = \dim \Derivs(F)_2$.
See \cite{MR1201387} for very interesting geometry arising from Waring decompositions of general quartics in $3$ variables.

Note that the rank of a general form of degree $4$ in $n \geq 7$ variables is strictly greater than $\dim \Derivs(F)_2$.
For $n=7$, the rank of a general $4$-form is $30$ while the dimension of the space of derivatives is $28$.
\end{example}

\begin{example}
Equality holds in the catalecticant bound for general plane conics, quartics, sextics, and octics.
That is, if $F$ is a general form in $n=3$ variables of even degree $d = 2a$, $1 \leq a \leq 4$,
then $\rk(F) = \dim \Derivs(F)_a$.
(The conic case follows by the earlier example on Fermat-type and quadratic forms,
and the quartic case is already in the previous example.)
See \cite{MR1201387}, \cite[Theorem 1.7]{MR1780430} for more on the geometry of Waring decompositions of these curves.
\end{example}

\subsection{Catalecticant condition for additivity of Waring rank}
Now we prove our first main result, showing that
Questions \ref{q: rank additive} and \ref{q: decompositions additive} both have positive answers
when the bounds $\rk(F_i) \geq \dim \Derivs(F_i)_a$ are in fact equalities.

\begin{theorem}\label{thm: waring strassen additivity catalecticant bound}
Let $F_1(\bx_1),\dotsc,F_k(\bx_k)$ be homogeneous forms of degree $d$ in independent tuples of variables $\bx_1,\dotsc,\bx_k$
and let $F = F_1 + \dotsb + F_k$.
Suppose there is an $a$, $1 \leq a \leq d-1$, such that for each $i$, $1 \leq i \leq k$,
$\rk(F_i) = \dim \Derivs(F_i)_a$.
Then $\rk(F) = \dim \Derivs(F)_a = \rk(F_1) + \dotsb + \rk(F_k)$.

If in addition $a \geq 2$ then furthermore every Waring decomposition of $F$ is a sum of Waring decompositions of the $F_i$'s.
That is, if $F = \ell_1^d + \dotsb + \ell_r^d$ with $r=\rk(F)$, then
each $\ell_j$ only involves variables from one tuple $\bx_i$,
and the $\ell_j$ that involve variables from $\bx_i$ give a Waring decomposition of $F_i$.
\end{theorem}

\begin{proof}
Observe that $\Derivs(F)_a = \Derivs(F_1)_a \oplus \dotsb \oplus \Derivs(F_k)_a$.
(See, for example, \cite[\textsection 2.8]{MR3426613}.)
We have
\[
\begin{split}
  \rk(F) &\geq \dim \Derivs(F_1+\dotsb+F_k)_a \\
    &= \sum_{i=1}^k \dim \Derivs(F_i)_a \\
    &= \sum_{i=1}^k \rk(F_i) \\
    &\geq \rk(F_1+\dotsb+F_k).
\end{split}
\]
This shows that $\rk(F)$ is equal to $\dim \Derivs(F)_a$ and to the sum of the $\rk(F_i)$.

Next suppose $F = \sum_{j=1}^r \ell_j^d$.
The space $\Derivs(F)_a$ is contained in the span of the $\ell_j^a$,
and since $\Derivs(F)_a$ has dimension equal to the number $r$ of the $\ell_j^a$,
it follows that for every $1 \leq j \leq r$, $\ell_j^a \in \Derivs(F)_a$.

Since $\Derivs(F)_a$ is the direct sum of the subspaces $\Derivs(F_i)_a$,
and all the forms in $\Derivs(F_i)_a$ involve only the variables in the tuple $\bx_i$,
the forms in $\Derivs(F)_a$ have no mixed terms involving variables from more than one tuple.
It follows that each $\ell_j$ involves variables from at most one tuple, otherwise, for $a \geq 2$, $\ell_j^a$ would have mixed terms.
\end{proof}

\begin{corollary}
Let $d \geq 3$ and let $n_1,\dotsc,n_k$ be positive integers,
$\bx_1,\dotsc,\bx_k$ independent tuples of variables with each $|\bx_i|=n_i$.
For each $i$, $1 \leq i \leq k$ let $F_i$ be a general $d$-form of rank at most $\binom{n_i + \lfloor d/2 \rfloor - 1}{n_i-1}$
in the variables $\bx_i$.
Then $\rk(F_1+\dotsb+F_k) = \sum \rk(F_i)$ and every
Waring decomposition of $F_1+\dotsb+F_k$ is a sum of Waring decompositions of the $F_i$.
\qed
\end{corollary}

Pedro Macias Marques and Elisa Postinghel very generously shared the following ideas with me.
The first statement is a partial converse to Theorem~\ref{thm: waring strassen additivity catalecticant bound}.
\begin{proposition}[Macias Marques, Postinghel]\label{proposition: converse}
Let $F = F_1 + \dotsb + F_k$
where $F_1(\bx_1),\dotsc,F_k(\bx_k)$ are homogeneous forms of degree $d$ in independent tuples of variables $\bx_1,\dotsc,\bx_k$.
Suppose there is an $a$, $2 \leq a \leq d-1$, such that $\rk(F) = \dim \Derivs(F)_a$.
Then for each $i$, $1 \leq i \leq k$, $\rk(F_i) = \dim \Derivs(F_i)_a$.
\end{proposition}
\begin{proof}%
%\footnote{An earlier version of this paper had an incorrect proof written by the author.
%The current, correct proof is along the lines originally suggested by Macias Marques and Postinghel.}
%%%
%%% WRONG PROOF %%%
%%%
%For each $i$, $\dim \Derivs(F_i)_a \leq \rk(F_i)$, and
%\begin{multline*}
%  \sum_{i=1}^k \dim \Derivs(F_i)_a \leq \sum_{i=1}^k \rk(F_i) \leq \rk(F) \\
%    = \dim \Derivs(F)_a = \sum_{i=1}^k \dim \Derivs(F_i)_a.
%\end{multline*}
%So all the inequalities are in fact equalities.
%%%
%%% The problem is the second inequality.
%%% In fact \sum_{i=1}^k \rk(F_i) \geq \rk(F), rather than \leq
%%%
%%% Here we go with a correct proof:
Let $r = \rk(F) = \dim \Derivs(F)_a$ and let $F = \ell_1^d + \dotsb + \ell_r^d$ be a decomposition of length $r$.
As above, each $\ell_i^a \in \Derivs(F)_a = \bigoplus_{i=1}^k \Derivs(F_i)_a$.
Therefore each $\ell_i$ involves variables from only one tuple $\bx_j$,
and the decomposition of $F$ turns out to be a sum of decompositions of the $F_i$.
Then
\[
  \sum_{i=1}^k \rk(F_i) \leq \sum_{i=1}^k \dim \Derivs(F_i)_a = \dim \Derivs(F)_a = \rk (F) .
\]
Hence $\rk(F) = \sum_{i=1}^k \rk(F_i)$, and for each $i$, $\rk(F_i) = \dim \Derivs(F_i)_a$.
\end{proof}
Now we have the following extension of Theorem~\ref{thm: waring strassen additivity catalecticant bound}.
\begin{corollary}[Macias Marques, Postinghel]\label{corollary: extension}
Let $F = F_1 + \dotsb + F_k$
where $F_1(\bx_1),\dotsc,F_k(\bx_k)$ are homogeneous forms of degree $d$ in independent tuples of variables $\bx_1,\dotsc,\bx_k$.
Suppose there is an $a$, $2 \leq a \leq d-1$,
such that $\rk(F_1) = \dim \Derivs(F_1)_a + 1$ and, for $2 \leq i \leq k$, $\rk(F_i) = \dim \Derivs(F_i)$.
Then $\rk(F) = \dim \Derivs(F)_a + 1 = \rk(F_1) + \dotsb + \rk(F_k)$.
\end{corollary}
\begin{proof}
We have
\begin{multline*}
  \dim \Derivs(F)_a \leq \rk(F) \leq \rk(F_1) + \dotsb + \rk(F_k) \\
    = \left( \sum_{i=1}^k \dim \Derivs(F_i)_a \right) + 1 = \dim \Derivs(F)_a + 1.
\end{multline*}
The case $\rk(F) = \dim \Derivs(F)_a$ is ruled out since the converse given above would imply
$\rk(F_1) = \dim \Derivs(F_1)_a$.
So $\rk(F) > \dim \Derivs(F)_a$.
\end{proof}

\section{Further bounds for Waring rank}\label{section: further bounds}

We review a bound for Waring rank described in \cite{MR3506475}.
We show it is additive on forms in independent variables,
so we can use this bound to describe a sufficient condition for a positive answer to Question~\ref{q: rank additive}.
Then we recall a lower bound for Waring rank from \cite{MR2628829}
and observe that it is also additive on forms in independent variables.

Recall the following common notation (see, for example, \cite{MR2628829}, \cite{MR2279854}, \cite{Teitler:2014gf}):
Let $V$ be a finite dimensional vector space with basis $\bx = (x_1,\dotsc,x_n)$,
let $S = S(V) = \bbC[\bx]$ be the symmetric algebra on $V$ and the polynomial ring on $\bx$,
and let $T = S(V^*) = \bbC[\partial_1,\dotsc,\partial_n]$ be the symmetric algebra on $V^*$,
which we regard as a polynomial ring in the dual variables $\partial_i$.
We let $T$ act on $S$ by differentiation, each $\partial_i$ acting as $\frac{\partial}{\partial x_i}$.
A $d$-form $F \in S^d V = S_d$ defines a hypersurface in $\bbP T_1$ or for that matter a (honest) function
on the affine space $T_1$.
Let $F^\perp = \{ \Theta \in T : \Theta F = 0 \}$, the apolar (or annihilating) ideal of $F$.
Since $F$ is homogeneous, the ideal $F^\perp$ is homogeneous.
Note that $\Derivs(F) \cong T/F^\perp$ as a vector space and, with reversed grading, as a $T$-module.

\subsection{Affine subspace bound}

The following is the bound for Waring rank described by Carlini, Guo, and Ventura \cite{MR3506475}.
\begin{theorem}[\cite{MR3506475}]
Let $1 \leq p \leq n$ be an integer, $F$ be a form in $n+1$ variables $x_0,\dotsc,x_n$ and let $F_k$ denote $\partial F/\partial x_k$.
If for all $\lambda_k \in \mathbb{C}$, with $1 \leq k \leq p$,
we have $\rk(F_0 +  \sum_{k=1}^p \lambda_k F_k) \geq m$
and $F_1,F_2,\dotsc,F_p$ are linearly independent, then $\rk(F) \geq m + p$.
\end{theorem}
Before anything else, we present a slight generalization with a simple proof.
Our proof is not substantially different from the proof of \cite{MR3506475}, but we give a more coordinate-free presentation.

In the following a \defining{linear subspace} is a subspace in the usual sense, passing through the origin.
An \defining{affine subspace} is a translation of a linear subspace; it may or may not pass through the origin.
A \defining{non-linear affine subspace} is an affine subspace that does not pass through the origin.

Recall the familiar fact that every affine line meets at least one coordinate hyperplane.
A slight generalization of this will be useful.
\begin{proposition}
Let $V \subseteq \Span\{v_1,\dotsc,v_r\}$.
Suppose $W \subseteq V$ is a $p$-dimensional non-linear affine subspace.
Then there is a $w \in W$ that can be written as a linear combination of at most $r-p$ of the $v_i$.
\end{proposition}
\begin{proof}
%Suppose without loss of generality that $v_1,\dotsc,v_b$ is a basis for the span of the $v_i$, $b \leq r$.
%Write elements of $V$ in the usual way as column vectors with respect to this basis.
%We claim there is an element in $W$ with at most $b-p$ nonzero entries,
%i.e., at least $p$ zeros.
%It is clear if $\dim(W)=0$, i.e., $W = \{w\}$ is a single nonzero vector.
%If $\dim(W)>0$ then there is an entry taking nonconstant values on $W$.
%Setting this entry to $0$ gives a $(p-1)$-dimensional subspace of the span of $b-1$ of the $v_i$
%and the result follows by induction.
It is immediate if $p=0$ and follows by induction for $p > 0$.
\end{proof}

Now the Carlini--Guo--Ventura lower bound for Waring rank is the case $a=d-1$ of the following statement.
Let us denote $\minrank(W) = \min\{ \rk(G) : G \in W \}$.
\begin{theorem}\label{thm: cgv}
Let $F$ be a form of degree $d$.
Then for all $1 \leq a \leq d-1$,
\[
  \rk(F) \geq \max\{ \dim(W) + \minrank(W) : 0 \notin W \subset \Derivs(F)_a \} ,
\]
the maximum taken over all nonlinear affine subspaces.
\end{theorem}
\begin{proof}
If $F = \ell_1^d + \dotsb + \ell_r^d$ then $\Derivs(F)_a \subseteq \Span\{\ell_1^a,\dotsc,\ell_r^a\}$.
Every non-linear affine subspace $W \subset \Derivs(F)_a$ of dimension $p$ must contain a point which is a linear combination of
at most $r-p$ of the $\ell_i^a$, hence is a point of rank at most $r-p$.
Thus $\minrank(W) \leq r - \dim(W)$.
\end{proof}

\begin{remark}
The catalecticant bound is the special case where $W$ is an arbitrary non-linear affine hyperplane,
so $\dim W$ is equal to $\dim \Derivs(F)_a - 1$, and $\minrank(W) \geq 1$.
In particular, when equality holds in the catalecticant bound then equality holds in the bound of Theorem~\ref{thm: cgv}.
\end{remark}

Here are some examples where equality holds in the bound of Theorem~\ref{thm: cgv} but not in the catalecticant bound.
The first example is due to Carlini--Guo--Ventura and was in fact the motivating example for the development of their bound.

\begin{example}[\cite{MR3506475}]
Let $F = x(y_1^2 + \dotsb + y_{n-1}^2 + y_n x)$.
Let $F_0 = \partial F / \partial x = y_1^2 + \dotsb + y_{n-1}^2 + 2 y_n x$.
For $1 \leq i \leq n$ let $F_i = \partial F / \partial y_i$,
so $F_i = 2x y_i$ for $1 \leq i \leq n-1$ and $F_n = x^2$.
Note $F_0,\dotsc,F_n$ are linearly independent since no two of them share any monomials.
Let $W_0$ be the span of $F_1,\dotsc,F_n$ and let $W = F_0 + W_0$.
Elements of $W$ are of the form
\[
  y_1^2 + \dotsb + y_{n-1}^2 + x(2y_n + \lambda_1 y_1 + \dotsb + \lambda_{n-1} y_{n-1} + \lambda_n x)
\]
which can be rewritten as
\[
  y_1^2 + \dotsb + y_{n-1}^2 + xL
\]
where $y_1,\dotsc,y_{n-1},x,L$ are linearly independent linear forms.
So this quadratic form has rank $n+1$.
Hence $\rk F \geq 2n+1$.
It is easy to see that $\rk F \leq 2n+1$, see \cite{MR3506475}.
So $\rk F = 2n+1$ is determined and equality holds in the bound of Theorem~\ref{thm: cgv} with $a=2$.
\end{example}

\begin{example}\label{example: ternary quartics}
Kleppe \cite{Kleppe:1999fk} and De Paris \cite{MR3437175} showed that every ternary quartic has rank at most $7$,
and Kleppe showed that rank $7$ is attained by a ternary quartic of the form $y^2(x^2+yz)$,
consisting of a smooth conic plus a doubled tangent line.
We are in a position to give an alternative to Kleppe's proof that $F = y^2(x^2+yz)$ has rank $7$.
It is easy to see $\rk(F) \leq 7$ ($\rk(x^2 y^2) = 3$ and $\rk(y^3 z) = 4$).
Note that
\[
  \frac{\partial F}{\partial y} + \lambda \frac{\partial F}{\partial x} + \mu \frac{\partial F}{\partial z}
    = 2x^2 y + y^2(3z + 2\lambda x + \mu y) = 2x^2 y + y^2 L
\]
giving a $2$ dimensional affine space of ternary cubics of rank $5$.
So $\rk(F) \geq 7$.
Thus $\rk(F) = 7$ is determined and equality holds in the bound of Theorem~\ref{thm: cgv} with $a=3$.
\end{example}

\begin{remark}
The maximum ranks for ternary cubics, quartics, and quintics are each attained by binomials:
in the cubic case, it is well known that $x^2 y + y^2 z$ has rank $5$, the maximum;
for quartics, the binomial $x^2 y^2 + y^3 z$ is discussed in the example above;
and \cite[Theorem 18]{MR3536055} shows that the quintic binomial $x y z^3 + y^4 z$ attains the maximum rank.
Binomials are not the only forms that have the maximum value for rank.
Nevertheless it would be interesting to determine the ranks of all binomials.
\end{remark}

\begin{remark}\label{remark: other ternary quartics}
The above $x^2 y^2 + y^3 z$ is not the only ternary quartic of rank $7$; see \cite[Theorem~3.6]{Kleppe:1999fk},
which asserts that ternary quartics $F$ with a certain property have rank $7$.
The condition is that $(F^\perp)_2$ must be $1$-dimensional and spanned by the square of a linear form.
Kleppe does not give an example of such a form, and it is perhaps not immediately obvious that such a form exists.
For the above $F = y^2(x^2+yz)$, $(F^\perp)_2$ contains the square $\partial_z^2 = \tfrac{\partial^2}{\partial z^2}$,
but $(F^\perp)_2$ is $3$-dimensional.
Alessandro De Paris kindly pointed this out to me along with the (not unique) example $G = (x+y)^4 + (x^3 + y^3)z$.
It is easy to see that $\rk(G) \leq \rk((x+y)^4) + \rk(x^3 z - z^3) + \rk(y^3 z + z^3) = 1 + 3 + 3 = 7$.
And one can check that $(G^\perp)_2$ is spanned by $\partial_z^2 = \tfrac{\partial^2}{\partial z^2}$,
so Kleppe's theorem shows $\rk(G) = 7$.
Alternatively, $\rk(G) \geq 7$ follows from \cite[Corollary 8]{MR3536055}.
\end{remark}

One can easily check that equality does not hold in the catalecticant bound in the above examples.

\begin{theorem}
Let $d \geq 3$, let $F_1(\bx_1),\dotsc,F_k(\bx_k)$ be $d$-forms in independent tuples of variables $\bx_1,\dotsc,\bx_k$,
and let $F = F_1 + \dotsb + F_k$.
Suppose that for each $i$, $1 \leq i \leq k$,
equality holds in the bound of Theorem~\ref{thm: cgv} for $\rk(F_i)$ with $a=2$.
Then equality holds in the bound of Theorem~\ref{thm: cgv} for $\rk(F)$ with $a=2$,
and $\rk(F) = \rk(F_1) + \dotsb + \rk(F_k)$.
\end{theorem}
\begin{proof}
For each $i$ let $W_i \subset \Derivs(F_i)_2$ be a $p_i$-dimensional nonlinear affine subspace and let $m_i$ be an integer such that
$\rk(G) \geq m_i$ for all $G \in W_i$, and $m_i + p_i = \rk(F_i)$.
Let $W = W_1 \times \dotsb \times W_k \subset \Derivs(F)_2$.
Let $p = \dim W = \sum \dim W_i = \sum p_i$.
Suppose $G = \sum G_i \in W$, each $G_i \in W_i$.
Since the $G_i$ are quadratic forms in independent variables $\bx_i$,
we have $\rk(G) = \sum \rk(G_i)$. 
So $\rk(G) \geq m = \sum m_i$.
Hence $\rk(F) \geq m + p = \sum (m_i + p_i) = \sum \rk(F_i) \geq \rk(F)$.
\end{proof}

If the symmetric version of Strassen's additivity conjecture holds for all forms in degree $a$,
then $2$ in the above theorem can be replaced by $a$.

\subsection{Singularities}
In this section we observe that
additivity of Waring ranks is implied by equality in a lower bound for Waring rank in terms of singularities
given in \cite{MR2628829}.

For $F$ and for $0 \leq a \leq d-1$, we let
$\Sigma_a(F) = \{ p \in T_1 : \mult_p(F) > a \}$,
the set of points at which $F$ vanishes with multiplicity strictly greater than $a$.
This is an algebraic set defined by the common vanishing of the $a$th derivatives of $F$,
but we will ignore the scheme structure.
Note $\Sigma_0(F)$ is the affine cone over the hypersurface defined by $F$
and $\Sigma_1(F)$ is the singular locus of this cone.
Recall that $F$ is called \defining{concise with respect to $\bx$}, or simply \defining{concise},
if $F$ cannot be written as a form in fewer variables,
even after a linear change of coordinates: explicitly, if $F \in S^d W$ for some $W \subseteq V$, then $W = V$.
The following are equivalent: $F$ is concise; the projective hypersurface defined by $F$ is not a cone;
$\Sigma_{d-1}(F)$ is supported only at the origin in $T_1$;
$(F^\perp)_1 = 0$; $\Derivs(F)_1 = S_1$.
See \cite{MR2279854}, \cite[\textsection2.2]{MR3426613}.
Note that in the context of practical computation the last two conditions are easily checked by linear algebra.

Now the lower bound for Waring rank in terms of singularities is the following.
\begin{theorem}[\cite{MR2628829}]
Suppose $F \in S^d V$ is concise and $1 \leq a \leq d-1$.
Then $\rk(F) \geq \dim \Derivs(F)_{d-a} + \dim \Sigma_a(F)$.
\end{theorem}
\noindent
(In \cite{MR2628829} the notation $\Sigma_a(F)$ refers to the projective locus which leads to a $+1$ in the statement of the bound.)
Here are some forms $F$ which attain equality in this bound:
\begin{enumerate}
\item $x_1 y_1 z_1 + \dotsb + x_n y_n z_n$ \cite[Proposition 7.1]{MR2628829}
\item $x(y_1^2 + \dotsb + y_n^2)$ and $x(y_1^2 + \dotsb + y_n^2 + x^2)$ \cite[Proposition 7.2]{MR2628829}
\end{enumerate}
When equality holds, Strassen's additivity conjecture follows.
\begin{theorem}
Suppose that $\bx_1,\dotsc,\bx_k$ are independent tuples of variables.
For each $i$, $1 \leq i \leq k$, let $V_i$ be the vector space with basis $\bx_i$.
For each $i$, $1 \leq i \leq k$, let $F_i(\bx_i)$ be a $d$-form which is concise with respect to $\bx_i$.
Let $F = F_1 + \dotsb + F_k$ and let $V = \bigoplus_{i=1}^k V_i$.
We consider $\Sigma_a(F_i) \subset V_i$ for each $i$ and $\Sigma_a(F) \subset V$.

Let $1 \leq a \leq d-1$
be such that for each $i$,
\[
  \rk(F_i) = \dim \Derivs(F_i)_{d-a} + \dim \Sigma_a(F_i) .
\]
Then $F = F_1 + \dotsb + F_k$ is concise with respect to $\bx = (\bx_1,\dotsc,\bx_k)$
and
\[
  \rk(F) = \dim \Derivs(F)_{d-a} + \dim \Sigma_a(F) = \sum_{i=1}^k \rk(F_i).
\]
\end{theorem}
\begin{proof}
Note that $\Sigma_a(F)$ is the Cartesian product of the $\Sigma_a(F_i)$.
So $\dim \Sigma_a(\sum F_i) = \sum \dim \Sigma_a(F_i)$.
The rest is just as before.
\end{proof}

\subsection{Other ranks}
See \cite{Teitler:2014gf} for an exposition of other ranks which are generalizations of Waring rank or variations on Waring rank,
such as simultaneous rank and multihomogeneous rank.
In each case the notion of rank is subadditive and has various lower bounds that are additive on forms (or linear spaces of forms)
in independent variables; we leave it to the reader to formulate sufficient conditions for the analogous additivity conjectures
for these generalized ranks, similar to the above results.

\section{Cactus rank}\label{section: cactus rank}

We recall the notion of cactus rank and
give a sufficient condition for a positive answer to the analogous version of Question~\ref{q: decompositions additive}.

\begin{definition}
Let $F \in S_d$ be a $d$-form.
A closed scheme $Z \subset \bbP V = \bbP S_1$ is \defining{apolar to $F$}
if the saturated ideal $I = I(Z)$ satisfies $I \subseteq F^\perp$.
\end{definition}
\begin{definition}
For a closed subscheme in projective space $X \subseteq \bbP W$
the \defining{span} of $X$, denoted $\Span(X) \subseteq \bbP W$, is the reduced projective linear subspace spanned by $X$,
that is, the smallest reduced projective linear subspace containing $X$ as a subscheme.
Equivalently, $\Span(X)$ is the projective linear subspace defined by the vanishing of the linear forms in the ideal $I(X)$.
We write $\Span'(X) \subseteq W$ for the affine cone over $\Span(X)$.
\end{definition}

The Apolarity Lemma states that $Z$ is apolar to $F$ if and only if the point $[F] \in \bbP S_d$
lies in the linear span of the scheme $\nu_d(Z)$, where $\nu_d : \bbP S_1 \to \bbP S_d$ is the Veronese map.
See, for example, \cite[pg.~280]{Pascal:1910sf} (and references therein),
\cite[Theorem 5.3]{MR1735271}, \cite[\textsection1.3]{MR1780430}, \cite[\textsection4.1]{Teitler:2014gf}.

When $Z = \{[\ell_1],\dotsc,[\ell_r]\}$ is a reduced zero-dimensional scheme
it follows that $Z$ is apolar to $F$ if and only if $F = \sum c_i \ell_i^d$ for some scalars $c_i$.
So the Waring rank $\rk(F)$ is equal to the minimum degree of a reduced zero-dimensional apolar scheme to $F$.

\begin{definition}
The \defining{cactus rank} of a form $F$, denoted $\crk(F)$, is the least degree of a zero-dimensional apolar scheme to $F$,
see \cite{MR2842085}.
This is also called the \defining{scheme length} of $F$, see \cite[Definition~5.1, Definition~5.66]{MR1735271}.
\end{definition}
Recall that cactus rank is bounded below by $\dim \Derivs(F)_a$ and is subadditive:
$\crk(F) \geq \dim \Derivs(F)_a$ and $\crk(F_1 + F_2) \leq \crk(F_1)+\crk(F_2)$.
Briefly, if $Z$ is apolar to $F$ then $\deg(Z) \geq \dim(T/I)_a \geq \dim (T/F^\perp)_a = \dim \Derivs(F)_a$,
and if $Z_i$ is apolar to $F_i$ for $i=1,2$, then $Z_1 \cup Z_2 \subset \bbP V_1 \cup \bbP V_2 \subset \bbP(V_1 \oplus V_2)$
is apolar to $F_1 + F_2$.

Evidently $\rk(F) \geq \crk(F)$.
So $\crk(F) = \dim \Derivs(F)_a$ occurs at least as often as the equality $\rk(F) = \dim \Derivs(F)_a$.
See Example~\ref{example: concentrated monomials cactus rank} for monomials $F$ such that
$\rk(F) > \crk(F) = \dim \Derivs(F)_a$.

\begin{theorem}\label{theorem: cactus rank}
Suppose that $\bx_1,\dotsc,\bx_k$ are independent tuples of variables.
For each $i$, $1 \leq i \leq k$, let $V_i$ be the vector space with basis $\bx_i$.
For each $i$, $1 \leq i \leq k$, let $F_i(\bx_i)$ be a nonzero $d$-form in the variables $\bx_i$,
or equivalently $0 \neq F_i \in S^d V_i$.
Let $F = F_1 + \dotsb + F_k$ and let $V = \bigoplus_{i=1}^k V_i$.

Suppose that there exists an integer $1 \leq a \leq d-1$ such that for each $i$, $\crk(F_i) = \dim \Derivs(F_i)_a$.
Then
\[
  \crk(F) = \dim \Derivs(F)_a = \sum_{i=1}^k \crk(F_i) .
\]

If in addition $a \geq 2$, then furthermore every zero-dimensional apolar scheme to $F$ of degree $\crk(F)$
is a union of apolar schemes to the $F_i$.
That is, if $Z \subset \bbP V$ is any zero-dimensional apolar scheme to $F$ of degree $\deg(Z) = \crk(F)$,
then $Z$ is a union of apolar schemes to the $F_i$, meaning that $Z = \bigcup_{i=1}^k Z_i$,
where each $Z_i \subset \bbP V_i \subset \bbP V$ is apolar to $F_i$.
\end{theorem}

\begin{proof}
The first statement follows just as in previous theorems.
Briefly,
\begin{multline*}
  \crk(F) \leq \sum \crk(F_i) = \sum \dim \Derivs(F_i)_a \\
    = \dim \Derivs(F)_a \leq \crk(F) ,
\end{multline*}
because $\Derivs(F)_a = \bigoplus \Derivs(F_i)_a$, for $1 \leq a \leq d-1$.

Now suppose that $a \geq 2$.
And suppose that $Z \subset \bbP V$ is apolar to $F$ and $r = \deg(Z) = \crk(F)$.

For each $i$, $1 \leq i \leq k$, let $\bpartial_i$ be a tuple of dual variables to $\bx_i$.
Concretely, if $\bx_i = (x_{i,1},\dotsc,x_{i,n_i})$
then $\bpartial_i = (\partial_{i,1},\dotsc,\partial_{i,n_i})$
where $\partial_{i,j}$ acts as $\partial/\partial x_{i,j}$.
We identify the dual space $V_i^*$ with the subspace of $T_1$ spanned by $\bpartial_i$.
We write $V_i^\perp = \bigoplus_{j \neq i} V_j^*$.
Note that the variety in $\bbP V = \bbP S_1$ defined by $V_i^\perp$ is $\Zeros(V_i^\perp) = \bbP V_i$.

Let $I = I(Z)$.
Since $I$ is a one-dimensional ideal, $r = \dim(T/I)_e$ for $e \gg 0$.
Since $I$ is saturated, $\dim(T/I)_0 \leq \dim(T/I)_1 \leq \dotsb$.
So $r \geq \dim(T/I)_a$.
From $I \subset F^\perp$ we get $\dim(T/I)_a \geq \dim(T/F^\perp)_a$.
And we have just seen that $\dim(T/F^\perp)_a = \dim \Derivs(F)_a = r$.
It follows that $\codim I_a = \codim (F^\perp)_a$.
But $I_a \subseteq (F^\perp)_a$ by apolarity, hence $I_a = (F^\perp)_a$.
Again because $I$ is saturated, it follows that $(F^\perp)_e \subset I$ for all $e \leq a$.
In particular $(F^\perp)_2 \subset I$.
And finally for all $1 \leq i < j \leq k$, $V_i^* V_j^* \subset (F^\perp)_2 \subset I$.
Note that $\Zeros(\sum_{1 \leq i < j \leq k} V_i^* V_j^*) = \bigcup_{i=1}^k \bbP V_i \subset \bbP V$.
This shows $Z$ is contained in the disjoint union $\bigcup_{i=1}^k \bbP V_i$.

For $1 \leq i \leq k$ let $Z_i = Z \cap \bbP V_i$.
Then $Z = \bigcup_{i=1}^k Z_i$, a disjoint union.

All that is left is to verify that each $Z_i$ is apolar to $F_i$.
Since $V = \bigoplus_{i=1}^k V_i$, then $S^d V_1 \oplus \dotsb \oplus S^d V_k \subset S^d V$ is a direct sum.
The containment $Z_i \subset \bbP V_i$ means $\Span(\nu_d(Z_i)) \subseteq \Span(\nu_d(\bbP V_i)) = \bbP S^d V_i$.
It follows that $\Span'(\nu_d(Z_1)) + \dotsb + \Span'(\nu_d(Z_k)) = \Span'(\nu_d(Z))$ is a direct sum.
Since $F \in \Span'(\nu_d(Z))$, there is a unique decomposition $F = \sum F'_i$ where $F'_i \in \Span'(\nu_d(Z_i)) \subset S^d V_i$.
Since the $S^d V_i$ are a direct sum, it must be $F'_i = F_i$.
So $[F_i] \in \Span(\nu_d(Z_i))$, that is, the schemes $Z_i$ are apolar to $F_i$, for each $i$.
This shows that $Z$ is a disjoint union of schemes apolar to the $F_i$, as desired.
\end{proof}

\begin{example}\label{example: concentrated monomials cactus rank}
Let $M = x_1^{a_1} \dotsm x_n^{a_n}$ be a monomial with $0 < a_1 \leq \dotsb \leq a_n$.
The Waring rank $\rk(M) = (a_2+1) \dotsm (a_n+1)$ was found by Carlini--Catalisano--Geramita \cite{MR2966824},
who also found that sums of monomials in independent variables satisfy the symmetric Strassen additivity conjecture.
(See also \cite{MR3017012}.)
Slightly earlier, the cactus rank $\crk(M) = (a_1+1)\dotsm(a_{n-1}+1)$ was found by Ranestad--Schreyer \cite{MR2842085}.
But they did not consider the cactus rank of a sum of monomials in independent variables,
and in fact this value is not known in general.
We can now answer this for the special case of ``concentrated'' monomials (defined below).

(Note that if $n=1$ so that $M = x^d$, then $\rk(M) = \crk(M) = 1$
and the given expressions for both Waring rank and cactus rank remain valid when they are interpreted as empty products.)

Say that a monomial $M$ with notation as above is \defining{concentrated} if $a_1 + \dotsb + a_{n-1} \leq a_n$,
equivalently if $a_n \geq \frac{d}{2}$, where $d = \deg(M)$.
Every monomial in $1$ or $2$ variables is concentrated.
One can check that for $a_1 + \dotsb + a_{n-1} \leq \delta \leq a_n$ we have
$\dim \Derivs(M)_{\delta} = (a_1+1) \dotsm (a_{n-1}+1) = \crk(M)$
(see \cite[Lemma 11.4]{MR2628829}).
In particular if $M$ is concentrated then $\crk(M) = \dim \Derivs(M)_{a}$ for $a = \lfloor \frac{d}{2} \rfloor$.

Thus if $F_1,\dotsc,F_k$ are concentrated monomials of degree $d$ in independent variables it follows that
$\crk(\sum F_i) = \sum \crk(F_i)$ and every zero-dimensional apolar scheme to $\sum F_i$
is a union of apolar schemes to the individual $F_i$.
\end{example}

\begin{example}\label{example: concentrated monomials border rank}
Recall that the \defining{border rank} $\brk(F)$ of a form is the least $r$ such that $[F]$ lies
in the Zariski closure of the locus of forms of Waring rank at most $r$, that is, the $r$th secant variety of the Veronese;
see, for example, \cite{MR2628829}, \cite[Definition 5.66]{MR1735271}.
In general border rank does not necessarily correspond to the length of any apolar scheme; see \cite{MR3333949}
for an example of a form $F$ such that $\brk(F) < \crk(F)$.

It is known that $\brk(F) \geq \dim \Derivs(F)_a$ for $0 \leq a \leq d$,
see, for example, \cite[Proposition 5.67]{MR1735271}, and that $\brk(F+G) \leq \brk(F) + \brk(G)$.
It immediately follows that if $F_1,\dotsc,F_k$ are $d$-forms in independent variables and $0 \leq a \leq d$ is such that
$\brk(F_i) = \dim \Derivs(F_i)_a$ for each $i$, then $\brk(\sum F_i) = \sum \brk(F_i)$.

It has been known for some time that for a monomial $M$,
\[
  \brk(M) \leq \crk(M) = (a_1+1)\dotsm(a_{n-1}+1),
\]
see \cite[Theorem 11.2]{MR2628829},
and that equality holds when $M$ is concentrated \cite[Theorem 11.3]{MR2628829}.
In particular if $F_1,\dotsc,F_k$ are concentrated monomials of degree $d$ in independent variables then
$\brk(\sum F_i) = \sum \brk(F_i) = \sum \crk(F_i)$.

However, Oeding has very recently shown that $\brk(M) = (a_1+1)\dotsm(a_{n-1}+1)$
for every monomial $M$ \cite{Oeding:2016kl}.
While it is not explicitly stated in \cite{Oeding:2016kl}, it seems to be the case that Oeding's technique
extends to show that if $F_1,\dotsc,F_k$ are any monomials of degree $d$ in independent variables then
$\brk(\sum F_i) = \sum \brk(F_i)$.
\end{example}

\section*{Acknowledgments}

I would like to thank Hailong Dao, Luke Oeding, and Kristian Ranestad for several helpful comments.
I am especially grateful to Jaros{\l}aw Buczy\'nski, Alessandro De Paris, Pedro Macias Marques, and Elisa Postinghel
for substantial and generous suggestions including the ideas behind
Proposition~\ref{proposition: converse}, Corollary~\ref{corollary: extension}, Remark~\ref{remark: other ternary quartics},
and simplifications in the proof of Theorem~\ref{theorem: cactus rank}.

\bigskip
\bibliographystyle{amsplain}
\renewcommand{\MR}[1]{}
%\bibliography{../../biblio}

\end{document}